\documentclass[12pt]{amsart}

\usepackage{amssymb,amscd, easybmat}
\usepackage{amsbsy,amsmath,amsfonts,amsthm,extarrows,bm}
\usepackage[english]{babel}
\usepackage{graphicx}
\usepackage[all]{xy}
\usepackage{xfrac}

\usepackage[utf8]{inputenc}
\usepackage[T1]{fontenc}

\usepackage{xr}
\usepackage[bookmarksopen=false,breaklinks=true,backref=page,pagebackref=true,plainpages=false, hyperindex=true,pdfstartview=FitH,colorlinks=true, pdfpagelabels=true,linkcolor=blue,citecolor=red,urlcolor=red,hypertexnames=false]{hyperref}

\newtheorem{theorem}{Theorem}[section]
\newtheorem{lemma}[theorem]{Lemma}
\newtheorem{corollary}[theorem]{Corollary}
\newtheorem{proposition}[theorem]{Proposition}
\newtheorem{remark}[theorem]{Remark}
\newtheorem{question}[theorem]{Question}

\newtheorem{definition}[theorem]{Definition}
\newtheorem{example}[theorem]{Example}

\newtheorem{Atheorem}{Theorem}[section]

\newtheorem{Acorollary}[Atheorem]{Corollary}

\newtheorem{Aproposition}[Atheorem]{Proposition}

\newenvironment{sproof}[1]
{\begin{proof}[#1]} {\end{proof}}

\newcommand{\Z}{\mathbb Z}

\newcommand{\Q}{\mathbb Q}

\newcommand{\GL}{\textnormal{GL}}
\newcommand{\SL}{\textnormal{SL}}

\newcommand{\sr}{\textnormal{sr}}
\newcommand{\glr}{\textnormal{glr}}

\newcommand{\E}{\textnormal{E}}

\newcommand{\Aut}{\textnormal{Aut}}

\newcommand{\Cl}{\operatorname{Cl}}
\newcommand{\V}{\operatorname{V}}

\newcommand{\Jac}{\mathcal{J}}
\newcommand{\Pic}{\operatorname{Pic}}

\newcommand{\adj}{\operatorname{adj}}
\newcommand{\trace}{\operatorname{trace}}
\newcommand{\Tr}{\operatorname{Tr}}

\newcommand{\nK}{\operatorname{N}_{K/\mathbb{Q}}}
\newcommand{\Norm}{\operatorname{N}}

\newcommand{\Fitt}{\operatorname{Fitt}}
\newcommand{\Rp}{R_{\mathfrak{p}}}
\newcommand{\Mp}{M_{\mathfrak{p}}}
\newcommand{\Nmm}{\mathcal{N}_{\mathbf{m}}}
\newcommand{\ovR}{\overline{R}}

\newcommand{\Ass}{\operatorname{Ass}}
\newcommand{\Min}{\operatorname{Min}}
\newcommand{\Rt}{\widetilde{R}}
\newcommand{\Det}{\operatorname{det}}
\newcommand{\content}{\operatorname{content}}

\newcommand{\PL}{\mathbb{P}^1}

\newcommand{\br}[1]{\lbrack #1 \rbrack}

\newcommand{\bb}{\mathbf{b}}
\newcommand{\eb}{\mathbf{e}}

\newcommand{\rb}{\mathbf{r}}
\newcommand{\cb}{\mathbf{c}}
\newcommand{\hb}{\mathbf{h}}

\newcommand{\vb}{\mathbf{v}}
\newcommand{\wb}{\mathbf{w}}
\newcommand{\mb}{\mathbf{m}}

\newcommand{\ann}{\operatorname{ann}}

\newcommand{\ia}{\mathfrak{a}}

\newcommand{\ip}{\mathfrak{p}}
\newcommand{\iq}{\mathfrak{q}}
\newcommand{\iO}{\mathfrak{O}}

\makeatletter
\def\paragraph{\@startsection{paragraph}{4}%
  \z@\z@{-\fontdimen2\font}%
  {\normalfont\bfseries}}
\makeatother

\makeatletter
\@namedef{subjclassname@2020}{%
  \textup{2020} Mathematics Subject Classification}
\makeatother

\title{Equivalent generating pairs of an ideal of a commutative ring}
\address{EPFL ENT CBS BBP/HBP. Campus Biotech. B1 Building, Chemin des mines, 9\\Geneva 1202, Switzerland}
\email{luc.guyot@epfl.ch}
\author{Luc Guyot}
\date{\today}

\begin{document}
\maketitle
\begin{abstract}
Let $R$ be a commutative ring with identity and let $I$ be a two-generated ideal of $R$. 
We denote by $\SL_2(R)$ the group of $2 \times 2$ matrices over $R$ with determinant $1$. 
We study the action of $\SL_2(R)$ by matrix right-multiplication on $\V_2(I)$, 
the set of generating pairs of $I$.
Let $\Fitt_1(I)$ be the second Fitting ideal of $I$. 
Our main result asserts that $\V_2(I)/\SL_2(R)$ identifies with a group of units of $R/\Fitt_1(I)$ 
via a natural generalization of the determinant if $I$ can be generated by two regular elements.
This result is illustrated in several Bass rings for which we also show that $\SL_n(R)$ acts transitively on 
$\V_n(I)$ for every $n > 2$. As an application, we derive a formula for the number of cusps of a modular group over a quadratic order.
\end{abstract}

\section{Introduction} \label{SecIntro}
Rings are supposed unital and commutative. 
The unit group of a ring $R$ is denoted by $R^{\times}$. Let $M$ be a finitely generated $R$-module. 
We denote by $\mu(M)$ the minimal number of generators of $M$.
For $n \ge \mu(M)$, we denote by $\V_n(M)$ the set of \emph{generating vectors} of $M$ of length $n$, i.e., the set of $n$-tuples in $M^n$ whose components generate $M$. 
We consider the action of $\GL_n(R)$ on $\V_n(M)$ by matrix right-multiplication. 
Let $\SL_n(R)$ be the subgroup of $\GL_n(R)$  of determinant $1$ matrices.
Let $\E_n(R)$ be the subgroup of $\SL_n(R)$ generated by the \emph{elementary matrices}, i.e., the matrices which differ from the identity by a single off-diagonal element. 
Let $G$ be a subgroup of $\GL_n(R)$. 
Two generating vectors $\mb$, $\mb' \in \V_n(M)$ are said to be \emph{$G$-equivalent}, which we also denote by $\mb' \sim_{G} \mb$, if there exists $g \in G$ such that 
$\mb' = \mb g$. 
Apart from Section \ref{SecLocalGlobal} where $M$ is any faithful two-generated module, $M$ will be a two-generated ideal of $R$.
Our chief concern is the description of the quotient $\V_2(M)/G$ with $G \in \left\{\SL_2(R) , \GL_2(R)\right\}$. 
In some applications, we also address the case $G = \SL_n(R)$ with $n > 2$, see Propositions \ref{PropENWithNGe3} and \ref{PropSL}.

Let us highlight earlier results pertaining to this topic.
When $M = R$, the elements of $\V_n(M)$ are called unimodular rows of size $n$. 
The orbit set $\V_n(R)/\GL_n(R)$ is in one-to-one correspondence with the isomorphism classes of the stably free $R$-modules of type $1$ \cite[Proposition I.4.8]{Lam06}.
It is easily seen that $\GL_2(R)$ acts transitively on $\V_2(R)$. For $n > 2$, the question as to whether $\GL_n(R)$ acts transitively on $\V_n(R)$ is central in the resolution of 
Serre's problem on projective modules \cite[Corollary I.4.5]{Lam06}. 
Das, Tikader and Zinna have described the van der Kallen Abelian group structure of $\V_{n + 1}(R)/\SL_{n + 1}(R)$ \cite{vdK83}
for $R$ a smooth affine real algebra of dimension $n \ge 2$ \cite[Theorem 1.2]{DTZ18}; see also \cite{Fas11} for seminal results of this flavor.
 
Although no result in this paper refers directly to $\E_n(R)$, let us mention 
that $\V_n(R)/\E_n(R)$ has been extensively studied and can be endowed with the structure of an Abelian group
if $R$ is of finite stable rank and $n$ is large enough \cite[Section VIII.5]{Lam06} . 
For $R = S\br{x_1, \dots, x_m}$, the ring of polynomials in $m$ indeterminates over a ring $S$, 
Murthy has determined conditions under which $\GL_n(R)$ acts transitively on $\V_n(\sum_{k = 1}^r Rx_i)$ for $r \le n \le m$  \cite{Mur03}. 

To introduce further results, let us define the \emph{determinant} $\det(M)$ of a finitely generated $R$-module $M$ as the exterior product $\bigwedge^{\mu}M$ where $\mu = \mu(M)$. 
The module $\det(M)$ is a cyclic $R$-module whose annihilator is $\Fitt_{\mu - 1}(M)$, the $\mu$-th Fitting ideal of $M$ (see Section \ref{SubSecFittAndDet} for definitions).
The \emph{determinant map} $\det: M^n \rightarrow \bigwedge^n M$ is defined by $\det(m_1, \dots, m_{\mu}) = m_1 \wedge \cdots \wedge m_n$. 
It is immediate to check that $\det(\V_{\mu}(M)) \subseteq \V_1(\det(M))$. 
For $R$ a quasi-Euclidean ring and $M$ an arbitrary finitely generated $R$-module, 
we proved that the determinant map induces a bijection from $\V_{\mu}(M)/\E_{\mu}(R)$ onto $\V_1(\det(M))$ \cite[Theorem A]{Guy17a}. 
In this paper, we aim for a similar description of $\V_2(I)/\SL_2(R)$ for $I$ a two-generated ideal of an arbitrary ring $R$. 

To state our main theorem, we need the following definitions.
An element of a ring $R$ is said to be \emph{regular} if it is not a zero-divisor. 
An ideal of $R$ is \emph{regular} if it contains a regular element.
The \emph{total ring of quotients $K(R)$} of $R$ is the localization $R[S^{-1}]$ where $S$ is the set of regular elements of $R$. 
An $R$-submodule $I$ of $K(R)$ is said to be a \emph{fractional ideal} of $R$ if there is a regular element $d$ of $R$ such that $dI \subseteq R$.
For $I$ a fractional ideal of $R$, we set $I^{-1} \Doteq \{ x \in K(R) \, \vert \, x I \subseteq R \}$.

\begin{Atheorem}[Theorem \ref{ThSL2AndFitt1I}] \label{ThSL2} 
Let $I$ be a two-generated ideal of a ring $R$ such that $\Fitt_1(I) = II^{-1} $, e.g., $I$ is generated by two regular elements. 
Then the determinant map induces an injection from $\V_2(I)/\SL_2(R)$ into $\V_1(\det(I)) \simeq (R/\Fitt_1(R))^{\times}$.
\end{Atheorem}

We investigate in Section \ref{SecConditions} conditions under which the determinant map induces a bijection. The main outcome is

\begin{Acorollary}[Corollary \ref{CorCompleteInvariant1}] \label{CorSL2}
Let $R$ be a ring whose quotients by regular ideals have stable rank $1$. Let $I$ be an ideal of $R$ which can generated by two regular elements. 
Then the determinant map induces a bijection from $\V_2(I)/\SL_2(R)$ onto $\V_1(\det(I))$.
\end{Acorollary}

In the remainder of this introduction, we shall illustrate 
consequences of Theorem \ref{ThSL2} in several Bass rings (see Section \ref{SecBassRings} for definition and background). 
In such rings, every ideal can be generated by two elements.

Our first application generalizes a theorem of Maass \cite[Proposition I.1.1]{vdGeer88} asserting 
that the number of cusps of the Hilbert modular group $\text{PSL}_2(\iO)$, with $\iO$ the ring of integers of a 
totally real algebraic number field, is equal to the class number of 
$\iO$. To introduce this result, further definitions are required.

The \emph{ring of mutlipliers} $\varrho(I)$ of a factional ideal $I$ of $R$ is the overring $\{ x \in K(R) \, \vert \, x I \subseteq I \}$.
The \emph{ideal class} $[I]$ of a fractional ideal $I$ of $R$ is 
the equivalence class of $I$ for the following relation on fractional ideals:  
$I \sim J$ if there are two regular elements $r, s \in R$ such that $rI = sJ$. 
We denote by $\Cl(R)$ the monoid of the ideal classes of $R$. 
We define the \emph{projective line} $\PL(K(R))$ over $K(R)$ as the quotient of $K(R) \times K(R) \setminus \{(0, 0)\}$ 
by the equivalence relation defined through: $\vb  \sim \wb$ if there are two regular elements $r, s \in R$ such that $r \vb = s \wb$.
There is a natural action of $\SL_2(R)$ on $\PL(K(R))$ from the right. 
The set of \emph{cusps of $R$} is the orbit set $\PL(K(R)) / \SL_2(R)$.

\begin{Acorollary}[Propositions \ref{PropP1} and \ref{PropEpsilonI}] \label{CorP1}
Let $R$ be an order in a number field. Then the number of cusps of $R$ is 
$$
\sum_{[I] \in \Cl(R), \, \mu(I) \le 2} [ (R/\Fitt_1(I))^{\times} : \epsilon(I) ].
$$
where $\epsilon(I)$ is the natural image of $\varrho(I)^{\times}$ in $(R/\Fitt_1(I))^{\times}$, see Definition \ref{DefEpsilonI}.
\end{Acorollary}
The finiteness of $\Cl(R)$ for $R$ an order in a number field is a direct consequence of  \cite[Theorem 2.6.3]{BS66}. 
As a result, such an order has finitely many cusps, which is a special case of \cite[Theorem 5]{Shi63}.
We use Corollary \ref{CorP1} to effectively compute the number of cusps of a quadratic order. To state this result, we shall denote by $\Pic(R)$, 
the \emph{Picard group} of $R$, that is, the group of the invertible elements of $\Cl(R)$.
 
 \begin{Atheorem} \label{ThOrder}
Let $\iO$ be the ring of integers of a quadratic number field.
Let $f > 0$ be a rational integer and let  $\iO_f$ be the unique order of index $f$ in $\iO$.
Then the number of cusps of $\iO_f$ is $$\sum_{f' \vert f} \frac{\varphi(f/f')}{2^{\epsilon(f, f')}} \vert \Pic(\iO_{f'}) \vert$$ where $\varphi$ is the Euler totient function and 
$\epsilon(f, f') = 1$ if $ \iO_{f'}$ has a unit of norm $-1$ and $f / f' > 2$, else $\epsilon(f, f') = 0$.
\end{Atheorem} 

By \cite[Theorem 12.12]{Neu99}, we have 
$$
\vert \Pic(\iO_{f'}) \vert = 
\frac{\vert \Pic(\iO) \vert}{\vert \iO^{\times} / \iO_{f'}^{\times}\vert} 
\frac{\vert ( \iO/\iO f')^{\times}  \vert}{\vert( \iO_{f'}/\iO f')^{\times} \vert}
$$ 
and moreover $\vert \Pic(\iO) \vert$ divides $\vert \Pic(\iO_{f'}) \vert$. 
Thus $\vert \Pic(\iO) \vert$, which is the number of cusps of $\iO$, divides the number of cusps of $\iO_f$. 
Note that the computation of $\epsilon(f, f')$ is a well-studied and solved problem, see e.g., \cite[Theorem 2]{Fur59}, \cite{Jen62} and the seminal work of Rédei \cite{Red53}.

Theorem \ref{ThOrder} relies on the computation of $\Fitt_1(I)$ for an arbitrary ideal $I$ of $R = \iO_f$.
This computation allows us moreover to describe 
$$\V_2(I)/\SL_2(R), \Aut_R(I) \backslash \V_2(I)/\SL_2(R) \text{ and } \V_2(I)/\GL_2(R)$$
and to show that $\SL_n(R)$ acts transitively on $\V_n(I)$ for every $n > 2$ in Section \ref{SecOrders} below. 
The results of Section \ref{SecOrders} enable us to solve in \cite{Guy20c} a problem in combinatorial group theory, 
namely the description of the Nielsen equivalence classes of generating vectors in semi-direct products of the form $\Z^2 \rtimes \Z$.

We present now an analysis similar to the one of Section \ref{SecOrders} for another Bass domain: the coordinate ring $R$ of 
the curve defined by the equation $y^2 = x^n$ over an arbitrary field $K$ where $n \ge 3$ is an odd integer. 
It is well-known and easy to show that $R$ is isomorphic to $K[x^2, x^n]$. 
For this example, the computation of $(R / \Fitt_1(I))^{\times}$ yields also an explicit description of $\V_2(I) / \SL_2(R)$.
The following definition is required to state this last result. For $S \subseteq K(R)$, a ring containing $R$, 
the \emph{conductor $(R:S)$ of $R$ in $S$} is the $R$-ideal $\{r \in R \, \vert \, rS \subseteq R \}$.  

\begin{Aproposition}[Lemma \ref{LemRhoI} and Proposition \ref{PropSL}] \label{PropCurves}
Let $K$ be a field and let $R = K[x^2, x^n]$ where $n \ge 3$ is an odd rational integer.
Let $I$ be an ideal of $R$ and let $f$ be the smallest positive odd rational integer such that $x^f I \subseteq I$. 
Then the following hold.
\begin{itemize}
\item[$(i)$] $\varrho(I) = K[x^2] + K[x] x^{f - 1}$. 
\item[$(ii)$] $ \Fitt_1(I) = (R: \varrho(I))  =  K[x^2]  x^{n - f} + K[x] x^{n - 1}$.
\item[$(iii)$] The ring $R / \Fitt_1(I)$ is isomorphic to $K[x^2] / K[x^2] x^{n -f}$.
\item[$(iv)$] The group $\SL_k(R)$ acts transitively on $\V_k(I)$ for every $k > 2$.
\end{itemize}
\end{Aproposition}
Thanks to Corollary \ref{CorSL2} and Proposition \ref{PropCurves}.$iii$, we can identify $\V_2(I) / \SL_2(R)$ with the group  $(K[x^2] / K[x^2] x^{n - f})^{\times}$ 
which in turn is isomorphic to $K_{+}^{ \frac{n - f}{2} - 1} \times K^{\times}$ where $K_{+}$ is the additive group of $K$.

\paragraph{Final remarks}
Within the course of the proofs of Theorem \ref{ThOrder} and Proposition \ref{PropCurves}, we established a common fact that we couldn't generalize to arbitrary Bass domains:
$\Fitt_1(I) = (R : \varrho(I))$ for $I$ a faithful ideal of $R$. In a Bass ring, this implies that $\Fitt_1(I)$ 
lies in a finite set of conductors as $R$ has only finitely many overrings contained in its normalization \cite[Proposition 2.2.$ii$]{LW85}. 
We also noticed for each Bass ring example in this article the following properties:
\begin{itemize}
\item $\Fitt_1(\Fitt_1(I)) = \Fitt_1(I)$,
\item $\varrho(\Fitt_1(I)) = \rho(I)$.
\end{itemize}
\begin{question}
To which extent can these observations be generalized?
\end{question}

\paragraph{Layout} This paper is organized as follows. Section \ref{SecNotation} introduces notation and background results. 
It encloses some of the fundamental properties of the determinant map and of the Bass stable rank. 
Section \ref{SecThSL2} presents a local-global criterion to study the action of $\SL_2(R)$ and then the proof of Theorem \ref{ThSL2}.
Section \ref{SecConditions} presents several conditions under which the determinant map induces a bijection in Theorem \ref{ThSL2}.
It includes in particular Corollary \ref{CorSL2} and its proof.
Section \ref{SecP1} is dedicated to the proof of Corollary \ref{CorP1}.
Sections \ref{SecOrders} and \ref{SecCurves} address in details the orders of a quadratic number field and the affine domain $K[x^2, x^n]$.
These two sections contain respectively the proof of Theorem \ref{ThOrder} and of Proposition \ref{PropCurves}.

\paragraph{Acknowledgments}
We are grateful to Justin Chen, François Couchot and Jean Fasel for their encouragements. 
We are specially thankful to Henri Lombardi and Wilberd van der Kallen 
for helpful comments and references.

\section{Notation and background} \label{SecNotation}

We assume throughout that $R$ is a commutative unital ring and $M$ is a finitely generated $R$-module. 
If $\ip$ is a prime ideal of $R$, we denote by $\Rp$ the localization of $R$ at $\ip$. 
Similarly, we denote by $\Mp$ its localization at $\ip$, which we identify with $M \otimes_R \Rp$. Given an element $m$ in $M$, we abuse notation in denoting also by $m$ its image $m\otimes_R 1 \in \Mp$. If $M$ and $N$ are two submodules of a given $R$-module, we denote by $(M:N)$ the ideal consisting of the elements $r \in R$ such that $rN \subseteq M$. If $Rm$ and $Rm'$ are two cyclic submodules of $M$, we simply write $(m : m')$ for $(Rm : Rm')$. 

\paragraph{Matrices over $R$ and $M$}
 Given two $1 \times n$ row vectors $\mb = (m_1,\dots, m_n) \in M^n$ and $\rb = (r_1, \dots, r_n) \in R^n$, we denote by $\mb^{\top}$ and $\rb^{\top}$ the $n \times 1$ column vectors obtained by transposition and we define the dot products $\mb \rb^{\top} = \rb \mb^{\top} \Doteq \sum_i r_i m_i$. Based on these identities, the product of any two matrices over $M$ and $R$, with compatible numbers of rows and columns, is uniquely defined. 
We denote respectively by $0_{m \times n}$ and by ${1_n}$ the $m \times n$ zero matrix and the $n \times n$ identity matrix over $R$.

\subsection{Fitting ideals and the determinant module} \label{SubSecFittAndDet} 
The determinant map is the invariant on which our study of $\V_2(M) / \SL_2(R)$ hinges.
This map can be effectively computed by means of the second Fitting ideal.

\paragraph{Fitting ideals}
Let 
$$
F \xrightarrow[]{\varphi} G \rightarrow M \rightarrow 0 
$$
be an exact sequence such that $F$ and $G$ are free modules over $R$ and $G$ is finitely generated. Let $I_i(\varphi)$ the ideal of 
$R$ generated by the $i \times i$ minors of the matrix of $\varphi: F \rightarrow G$, agreeing that $I_0(\varphi) = R$.  
Then the \emph{$(i + 1)$-th Fitting ideal of $M$} is defined as $\Fitt_i(M) \Doteq I_{\mu(G) - i}(\varphi)$. 
The Fitting ideals are independent of $\varphi$ by Fitting's lemma \cite[Corollary 20.4]{Eis95}. 
They commute with \emph{base change}, i.e, $\Fitt_i(M \otimes_R S) = \Fitt_i(M) \otimes_R S$ for any $R$-algebra $S$ \cite[Corollary 20.5]{Eis95}, and hence with localization.

\paragraph{The determinant of a finitely generated module} 
Let $\mu = \mu(M)$.
Recall that the determinant $\det(M)$ of $M$ is the exterior product $\bigwedge^{\mu}M$. This is a cyclic $R$-module whose annihilator is $\Fitt_{\mu - 1}(M)$ \cite[Exercise 20.9.i]{Eis95}.
Given $\mb = (m_1, \dots, m_{\mu}) \in \V_{\mu}(M)$, we denote by $\phi_{\mb}: \det(M) \rightarrow R/\Fitt_{\mu - 1}(M)$ the $R$-isomorphism induced by the map $m_1 \wedge  \cdots  \wedge m_{\mu} \rightarrow 1 + \Fitt_{\mu - 1}(M)$. Given another generating pair $\mb' = (m_1', \dots, m_{\mu}') \in M^{\mu}$, we define 
$$\text{det}_{\mb}(\mb') \Doteq \phi_{\mb}(m_1' \wedge \dots \wedge m_{\mu}').$$
It is easily checked that the image of $\V_{\mu}(M)$ under $\det_{\mb}$ is a subgroup of $(R/\Fitt_{\mu - 1}(M))^{\times}$ which does not depend on the choice of $\mb$.

The following lemma is straightforward.

\begin{lemma} \label{LemDet}
Let  $\mb, \mb', \mb'' \in \V_{\mu}(M)$.
Then the following assertions hold.
\begin{itemize}
\item[$(i)$] $\det_{\mb}(\mb' A) = \det(A)\det_{\mb}(\mb')$, for every $\mu \times \mu$ matrix $A$ over $R$.
\item[$(ii)$] $\det_{\mb}(\mb'') = \det_{\mb}(\mb') \det_{\mb'}(\mb'')$.
\item[$(iii)$] $\det_{\mb}(\mb') = 1 + \Fitt_{\mu - 1}(M)$, if and only if, $\det_{\mb}(\mb') = 1 + \Fitt_{\mu - 1}(M_{\ip})$ 
for every maximal ideal $\ip$ of $R$, where $\mb$ and $\mb'$ denote (abusively) their natural images in $\V_{\mu - 1}(M_{\ip})$ in the left-hand side of the identity.
\end{itemize}
\end{lemma}

$\square$

When there is no risk of ambiguity, we simply denote by $1$ the identity element of $R/\Fitt_{\mu -1}(M)$.
The ideal $\Fitt_{\mu - 1}(M)$ has a convenient description which makes the computation of $\det_{\mb}$ 
effective when a workable presentation of $M$ is given. Given $\mb \in \V_{\mu}(M)$, we say that an element $r$ is \emph{involved in a relation of $\mb$} 
if there is $(r_1,\dots, r_{\mu}) \in R^{\mu}$ such that $ \sum_{i = 1}^{\mu} r_i m_i = 0$ and $r = r_i$ for some $i$. 

\begin{lemma} \label{LemFittMuMinusOne}
Let $\mb \in\V_{\mu}(M)$. Then $\Fitt_{\mu - 1}(M)$  is the set of the elements of $R$ involved in a relation of $\mb$. 
\end{lemma}
$\square$
 
Given  $\mb \in\V_{\mu}(M)$, let $\overline{\mb} = \pi(\mb)$ be the image of $\mb \in V_{\mu}(M)$ by the natural map $\pi : M \rightarrow M/\Fitt_{\mu - 1}(M)M$ 
 and let $\eb$ be the canonical basis of $(R/\Fitt_{\mu - 1}(M))^{\mu}$. Then the map $\overline{\mb} \mapsto \eb$ induces an isomorphism $\varphi_{\mb}$ from $M/\Fitt_{\mu - 1}(M)M$ onto $(R/\Fitt_{\mu - 1}(M))^{\mu}$.
The following lemma shows how the map $\pi$ can be used to compute $\det_{\mb}$.
\begin{lemma} \label{LemDetPi}
Let $\mb, \mb' \in \V_{\mu}(M)$. Then $\det_{\mb}(\mb')$ is the determinant of the unique $\mu \times \mu$ matrix $\overline{A}$ over $R/\Fitt_{\mu - 1}(M)$ 
satisfying $\pi(\mb') = \pi(\mb) \overline{A}$, that is the matrix of $\varphi_{\mb} \circ \pi(\mb')$ with respect to $\eb$.
\end{lemma}
$\square$

\subsection{Ranks} \label{SubSecRanks}
Following  \cite[Section 6.7.2]{McCR87}, we define the Bass stable rank of a finitely generated $R$-module $M$.
An integer $n > 0$ lies in the \emph{stable range of} 
$M$ if for every $\mb = (m_1,\dots,m_{n + 1}) \in \V_{n + 1}(M)$, there is $r \in R$ such that $(m_1 + r m_{n + 1}, \dots, m_{n} + r m_{n + 1})$ belongs to $\V_n(M)$. 
If $n$ lies in the stable range of $M$, then so does $k$ for every $k > n$ \cite[Lemma 11.3.3]{McCR87}.
The \emph{stable rank $\sr(M)$ of $M$} is the least integer in the stable range of $M$. 
Extending naturally the linear rank introduced in \cite[Section 11.3.3]{McCR87} to finitely generated modules, we define the \emph{linear rank $\glr(M)$} of $M$ 
as the least integer $n \ge \mu(M)$ such that $\GL_k(R)$, acts transitively on $\V_k(M)$ for every $k > n$. 
Note that the analogous definition based on $\SL_k(R)$ yields also the rank $\glr(M)$. 
It is easily checked that $\SL_2(R)$ acts transitively on $\V_2(R)$. Hence $\glr(R) = 2$ is equivalent to 
$\glr(R) = 1$. A ring $R$ is said to be a \emph{Hermite ring} if $\glr(R) = 1$.
A ring $R$ is \emph{semi-local} if it has only finitely many maximal ideals. 

\begin{proposition}  \label{PropStableRank}
 The following hold:
\begin{itemize}
\item[$(i)$] $\sr(R) = 1$ if $R$ is semi-local \cite[Corollary 10.5]{Bas64}.
\item[$(ii)$] $\sr(R) \le \dim_{\text{Krull}}(R) + 1$ \cite[Corollary 2.3]{Heit84} or \cite[Theorem 2.4]{CLQ04}.
\end{itemize}
\end{proposition}

A ring $R$ is \emph{K-Hermite} if for every $n \ge 2$ and every $\rb \in R^n$, there is $\gamma \in \GL_n(R)$ such that 
$\rb \gamma = (0,\dots, 0, d)$ for some $d \in R$.
The following lemma will be used to show that $\SL_n(R)$ acts transitively on $\V_n(I)$ for $R$ and $I$ as in Propositions \ref{PropENWithNGe3} and \ref{PropSL}.

\begin{lemma} \label{FreeOverKHermite}
Let $R$ be a ring which is a free module of rank $k$ over a subring $S$. Assume that $S$ is a K-Hermite ring.
Let $I$ be an ideal of $R$ such that $\mu(I) \le k$ and let $n > k$. Then for every $\mb \in \V_n(I)$, there is $\sigma \in \SL_n(R)$ such that
$\mb \sigma = (\mb', 0, \dots, 0)$ for some $\mb' \in \V_k(I)$.  
\end{lemma}

The proof is a straightforward induction on $k$ consisting mainly in unrolling definitions. It is therefore omitted.

\section{Generating pairs and the determinant map} \label{SecThSL2}

\subsection{A local-global principle} \label{SecLocalGlobal}

In this section $M$ denotes a two-generated $R$-module. 
We shall prove that the $\SL_2(R)$-equivalence of two generating pairs of $M$ is a local property when $M$ is faithful. 
This property is key to generalize Theorem \ref{ThSL2} to every finitely generated $R$-modules of a ring $R$ 
with some Dedekind-like properties \cite[Theorem A]{Guy20b}. 

\begin{proposition} \label{PropEqLocal}
Let $M$ be a two-generated faithful $R$-module. Let $\mb, \mb' \in \V_2(M)$. Then the following are equivalent:
\begin{itemize}
\item[$(i)$] $\mb \sim_{\SL_2(R)} \mb'$,
\item[$(ii)$] $\mb \sim_{\SL_2(R_{\ip})} \mb'$ for every maximal ideal $\ip$ of $R$.
\end{itemize}
\end{proposition}
As an immediate consequence of Proposition \ref{PropEqLocal}, we observe that $\SL_2(R)$ acts transitively on $\V_2(M)$ 
if $M$ is a two-generated faithful projective module of constant rank $1$. 
This well-known fact is instrumental in the proof of Maass theorem on the number of cusps of the Hilbert modular group \cite[Proposition I.1.1]{vdGeer88}.
More generally, we have

\begin{corollary}
Let $M$ be a two-generated locally cyclic $R$-module. If the natural homomorphism $\SL_2(R) \rightarrow \SL_2(R/\ann(M))$ is surjective, 
then $\SL_2(R)$ acts transitively on $\V_2(M)$.
\end{corollary}
$\square$

\begin{remark} \label{RemStableRank2}
The natural map $\SL_n(R) \rightarrow \SL_n(\overline{R})$ is surjective for every $n$ and every quotient $\overline{R}$ of $R$ if $\sr(R) \le 2$ \cite[Corollary 8.3]{EO67}. 
See also the partial stability result \cite[Theorem 8.3]{EO67} when $R$ is a ring of univariate polynomials over a principal ideal domain.
\end{remark}

The proof of Proposition \ref{PropEqLocal} relies essentially on

\begin{lemma} \label{LemSysLin}
Let $M$ be an $R$-module and let $A = (a_{ij})$ be an $m \times n$ matrix with coefficients in $M$. Let $b = (b_i)$ an $n \times 1$ 
column vector with coefficients in $R$. Then the following are equivalent:
\begin{itemize}
\item[$(i)$] The linear system $Ax =b$ has a solution in $R^m$,
\item[$(ii)$] For every maximal ideal $\ip$ of $R$, the linear system $Ax =b$ has a solution in $\Rp^m$.
\end{itemize}
\end{lemma}

Lemma \ref{LemSysLin} is a simple variation of \cite[Proposition 1]{HS86}. 
A proof is provided for the convenience of the reader.
\begin{sproof}{Proof of Lemma \ref{LemSysLin}}
The implication $(i) \Rightarrow (ii)$ is evident. Assuming that $(ii)$ holds true, we shall prove that so does $(i)$. 
For every prime ideal $\ip$, we can find by hypothesis $m$ elements $x_j(\ip)$ in $R$ with $j \in \{1, \dots, m\}$ and two elements $s(\ip), t(\ip) \in R \setminus \ip$, such that 
$t(\ip)\left(\sum_{j = 1}^m x_j(\ip)a_{ij} \right)  = s(\ip) t(\ip) b_i$ holds for every $i \in \{ 1, \dots, n \}$. 
Since the ideal of $R$ generated by all the elements $s(\ip)t(\ip)$ is not contained in any maximal ideal of $R$, 
these elements generate $R$. Therefore we can find $l$ elements $\lambda_1, \dots, \lambda_l \in R$ 
for some $l \ge 1$ and $l$ prime ideals $\ip_1, \dots, \ip_l$ such that $\sum_{k = 1}^l \lambda_k s(\ip_k)t(\ip_k) = 1$. 
Setting $x_j \Doteq \sum_{k = 1}^l \lambda_k t(\ip_k)x_j(\ip_k)$ yields a solution $x = (x_j)$ of the linear system $Ax = b$ in $R^m$.
\end{sproof}

If $A$ is a square matrix over $R$, we denote by $\adj(A)$ its \emph{adjugate matrix}, i.e., 
the transpose of the cofactor matrix of $A$. The adjugate matrix satisfies the identities $A \adj(A) = \adj(A) A = \det(A) 1_n$. 
If $A = \begin{pmatrix} a_{11} & a_{12} \\ a_{21} & a_{22} \end{pmatrix}$, then $\adj(A) = \begin{pmatrix} a_{22} & -a_{12} \\ - a_{21} & a_{11} \end{pmatrix}$.

\begin{sproof}{Proof of Proposition \ref{PropEqLocal}}
In order to prove the equivalence $(i) \Leftrightarrow (ii)$, it suffices to show that $\mb \sim_{G} \mb'$ is equivalent to the existence of a solution of a linear system as in Lemma \ref{LemSysLin}. 
Let $\sigma$ be a $2 \times 2$ matrix over $R$ such that $\mb \sigma = \mb'$. Since $M$ is faithful and $\mb$ generates $M$, the identity $\det(\sigma) = 1$ 
holds true if and only if $\mb \sigma\adj(\sigma) = \mb$, which is equivalent to $\mb'\adj(\sigma) = \mb$. 
Therefore $\mb \sim_{G} \mb'$ holds true if and only if the following linear system with coefficients in $M$
$$\left\{
\begin{array}{c}
\mb \sigma = \mb', \\
\mb'\adj(\sigma) = \mb
\end{array}\right.
$$
has a solution $\sigma \in R^4$.
\end{sproof}

\subsection{Proof of Theorem \ref{ThSL2}} \label{SecProofOfMainTh}

In this section, we shall prove Theorem \ref{ThSL2}, that is

\begin{theorem} \label{ThSL2AndFitt1I}
Let $I$ be a two-generated ideal of $R$ such that $\Fitt_1(I) = II^{-1} $, e.g., $I$ is generated by two regular elements. 
Let $\mb, \mb' \in \V_2(I)$.
Then the following are equivalent.
\begin{itemize}
\item[$(i)$] $\mb \sim_{\SL_2(R)} \mb'$,
\item[$(ii)$] $\det_{\mb}(\mb') = 1$.
\end{itemize}
In other words, the determinant map induces an injection from $\V_2(I)/\SL_2(R)$ into $\V_1(\det(I)) \simeq (R/\Fitt_1(R))^{\times}$.
\end{theorem}

The proof of Theorem \ref{ThSL2AndFitt1I} relies on the following two lemmas.

\begin{lemma} \label{LemImPhi}
Let $I$ be a two-generated ideal of $R$. 
Let $\mb, \mb' \in \V_2(I)$ and let $A$ be a $2 \times 2$ matrix over $R$ such that $\mb A = \mb'$. Let $\Nmm$ be the set of $2 \times 2$ matrices $N$ over $R$ such that $\mb N = (0, 0)$.
Let $\Phi_A: \Nmm \rightarrow R$ be the map defined by $\Phi_A(N) = \det(A + N) - \det(A)$. Then we have
$$II^{-1} \subseteq \Phi_A(\Nmm)  \subseteq \Fitt_1(I).$$
If $I$ is moreover faithful, then $\Phi_A(\Nmm)$ is an ideal of $R$.
\end{lemma}

\begin{lemma} \label{LemFitt1}
Let $I$ be an ideal generated by two regular elements of $R$.  Then $\Fitt_1(I) = II^{-1}$.
\end{lemma}

\begin{proof}
Let $\mb = (a, b) \in \V_2(I)$. By Lemma \ref{LemFittMuMinusOne}, the ideal $\Fitt_1(I)$ is generated by the components of the row vectors $\cb$ such that $\mb \cb^{\top} = 0$. Therefore 
\begin{eqnarray*}
\Fitt_1(I) & =  & \left(a :b\right) + \left(b : a\right) = \left(R \cap R\frac{a}{b}\right) + \left(R \cap R\frac{b}{a}\right) \\
              & = & \left(\frac{Ra \cap Rb}{ab}\right)\left(Ra + Rb\right) = J I.
\end{eqnarray*}
where $J \Doteq \frac{Ra \cap Rb}{ab}$.
Since by definition $I^{-1}  =   \left(\frac{1}{a}R\right) \cap \left(\frac{1}{b}R\right) = J$, the result follows.
\end{proof}

\begin{sproof}{Proof of Lemma \ref{LemImPhi}}
It is straightforward to check that $\Phi_A(N) =\det(N) + \trace(\adj(A)N)$. Let us fix $N \in \Nmm$ and an exact sequence
$$
F \xrightarrow[]{\varphi} R^2 \xrightarrow[]{\pi_{\mb}} I \rightarrow 0 
$$
where $F$ is a free $R$-module and $\pi_{\mb}$ is the $R$-module epimorphism mapping the canonical basis of $R^2$ to $\mb$.
Since $\mb N = 0$, the columns of $N$ are $R$-linear combinations of the columns of the matrix of $\varphi$. 
Therefore, we have $\det(N) \in I_2(\varphi)$ and $\trace(\adj(A)N) \in I_1(\varphi)$, which yields $\Phi_A(N) \in \Fitt_1(I)$.
Let us write $\mb = (a,b)$ and $\mb' = (a', b')$ and set 
$N(r,s) \Doteq \begin{pmatrix}
rb & sb \\
-ra & -sa
\end{pmatrix}
$ for $r,s \in I^{-1}$.
Using the identity $\mb' = \mb A$,  we easily infer that $\Phi_A(N(r,s)) = rb' - sa'$, which shows that $ II^{-1}  \subseteq \Phi_A(\Nmm)$.
 If in addition $I$ is faithful, then we have $\Fitt_0(I) = I_2(\varphi) = 0$ \cite[Proposition 20.7.a]{Eis95}. 
 As a result we get $\det(N) = 0$ and hence $\Phi_A(N) = \trace(\adj(A)N)$, so that $\Phi_A(\Nmm)$ is an ideal of $R$.
\end{sproof}

We are now in position to prove Theorem \ref{ThSL2AndFitt1I}.

\begin{sproof}{Proof of Theorem \ref{ThSL2AndFitt1I}}
The implication $(i) \Rightarrow (ii)$ is evident. 
Let us prove $(ii) \Rightarrow (i)$.
Since $\det_{\mb}(\mb') = 1$, there is a $2 \times 2$ matrix $A$ over $R$ such that $\mb A = \mb'$ and $\det(A) \in 1 + \Fitt_1(I)$.
As $II^{-1} = \Fitt_1(I)$, we infer from Lemma \ref{LemImPhi} that $\Phi_A(\Nmm) = \Fitt_1(I)$. Hence we can find a matrix $B$ such that 
$\mb B = \mb'$ and $\det(B) = 1$, which completes the proof. 
\end{sproof}

The following lemma provides a condition under which 
Theorem \ref{ThSL2AndFitt1I} applies without assuming the generators of $I$ to be regular.

\begin{lemma} \label{LemMaximalIdeal} Let $I$ be a two-generated maximal ideal of $R$ which is faithful but not projective of constant rank $1$. 
Then we have $\Fitt_1(I) = II^{-1} = I$.
\end{lemma}

\begin{proof}
Since $I \subseteq II^{-1} \subseteq \Fitt_1(I)$ holds for any two-generated ideal $I$ of $R$, it is sufficient to show that 
$\Fitt_1(I) \subsetneq R$ when $I$ is not projective of constant rank $1$. The latter follows from \cite[Propositions 20.7.a and 20.8]{Eis95}.
\end{proof}

We conclude this section with an example, namely the specialization of \cite[Corollary 1.2]{Mur03} to the case of two variables.

\begin{example}
Let $S = R[x, y]$ be the ring of bivariate polynomials over a ring $R$ and let $I = Sx + Sy$. 
As $x$ and $y$ are regular, Theorem \ref{ThSL2AndFitt1I} applies. It is easily checked that 
$\Fitt_1(I) = I$ so that the determinant map induces an injective map $\V_2(I) / \SL_2(S) \rightarrow R^{\times}$ 
which is clearly surjective. Thus $\V_2(I) / \SL_2(S) \simeq R^{\times}$.
\end{example}

\section{Conditions under which $\V_2(I)/\SL_2(R)$ identifies with $(R/\Fitt_1(I))^{\times}$} \label{SecConditions}

In this section, we investigate conditions under which the induced map of Theorem \ref{ThSL2} is a bijection.
This is certainly the case if the natural group homomorphism $R^{\times} \rightarrow (R/\Fitt_1(I))^{\times}$ is surjective. 
The following lemma derives a less restrictive condition under which the identification $\V_2(I)/\SL_2(R) \simeq (R/\Fitt_1(I))^{\times}$ can also be made.

\begin{lemma} \label{LemSurjectivity}
Let $I$ be a two-generated ideal of $R$. Assume that we can find $(a, b) \in \V_2(I)$ such that the natural group homomorphism 
$$(R/(a: b))^{\times} \rightarrow (R/\Fitt_1(I))^{\times}$$ is surjective 
Then $\det_{(a, b)}: \V_2(I) \rightarrow  (R/\Fitt_1(I))^{\times}$ is surjective.
\end{lemma}

\begin{proof}
Let $u \in   (R/\Fitt_1(I))^{\times}$. By hypothesis, we can find $r \in R$ such that $r + \Fitt_1(I) = u$ and $r + (a:b)$ is a unit of $\ovR \Doteq R/(a: b)$.
Let $\varphi: I/Ra \rightarrow \ovR$ be the $R$-isomorphism which maps $b + Ra$ to $1 + (a : b)$.
Since $\varphi(rb + Ra) = r + (a: b)$, we deduce that $I$ is generated by $(a, rb)$. We conclude the proof by observing that $\det_{(a, b)}(a, rb) = r \det_{(a, b)}(a, b) = u$.
\end{proof}

We say that the ring $R$ has property $(*)$ if every surjective ring homomorphism $R \twoheadrightarrow S$ 
induces a surjective group homomorphism $R^{\times} \twoheadrightarrow S^{\times}$.
If there is $(a, b) \in  \V_2(I)$ such that $R/(a: b)$ has $(*)$, then we infer from Lemma \ref{LemSurjectivity} that $\det_{(a, b)}$ is surjective. 
A ring $R$ is called  a \emph{semi-field} if the Krull dimension of $R/\Jac(R)$ is zero where $\mathcal{J}(R)$ denotes the Jacobson radical of $R$. 
Semi-fields enjoy property $(*)$ \cite[Exercise IX.6.16]{LQ15} \cite[Theorem 5.1]{Chen17a}. 
semi-local rings and von Neuman regular rings belong to the class of semi-fields, which is a subclass of the local-global rings, 
see e.g., \cite[Fact IX.6.2]{LQ15} and \cite[Examples 4.1 and 4.2 of Section V.4]{FS01}. 
Local-global rings have stable rank $1$ and it turns out that this property fully charaterizes rings with $(*)$:

\begin{proposition}{\cite[Lemma 6.1]{EO67}} \label{PropSurjectionOfUnitGroups}
The following are equivalent:
\begin{itemize}
\item[$(i)$] The ring $R$ has property $(*)$.
\item[$(ii)$] $\sr(R) = 1$.
\end{itemize}
\end{proposition}

\begin{proof}
Assume that the natural homomorphism $R^{\times} \rightarrow (R/\ia)^{\times}$ is surjective for every ideal $\ia$ of $R$. Let $(a, b) \in \V_2(R)$. 
Then $a + \ia$ is unit of $R/I$ with $\ia = Rb$. Therefore we can find $u \in R^{\times}$ such that $a + \ia = u + \ia$, i.e., there is $r \in R$ verifying $a + r b = u$. 
As a result $\operatorname{sr}(R) = 1$.
Let us assume now that $\operatorname{sr}(R) = 1$. Let $\ia$ be an ideal of $R$ and let $a + \ia \in (R/\ia)^{\times}$. 
Then we can find $b \in \ia$ such that $(a, b) \in \V_2(R)$. 
Because $R$ is of stable rank $1$, there is $r \in R$ such that $u \Doteq a + rb$ is a unit of $R$.
\end{proof}

We present now several criteria which guarantee that the induced map of Theorem \ref{ThSL2} is a bijection.

\begin{proposition} \label{PropSL2CompleteInvariant}
Let $I$ be a two-generated ideal of $R$ such that the two following hold:
\begin{itemize}
\item $\Fitt_1(I) = II^{-1}$.
\item  There is $(a, b) \in \V_2(I)$ such that the natural group homomorphism $$\left(R/(a: b)\right)^{\times} \rightarrow (R/\Fitt_1(I))^{\times}$$ is surjective.
\end{itemize}
Then $\det_{(a,b)}$ induces a bijection from $\V_2(I)/\SL_2(R)$ onto $(R/\Fitt_1(I))^{\times}$.
\end{proposition}

\begin{proof}
Combine Theorem \ref{ThSL2AndFitt1I} and Lemma \ref{LemSurjectivity}.
\end{proof}

We say that a ring $R$ is \emph{almost of stable rank $1$} if $R/Rr$ has stable rank $1$ for every regular element $r \in R$.
This definition adapts Mac Govern's definition \cite[Section 4]{McGov08} to allow the existence of zero divisors in $R$.
A ring of Krull dimension at most $1$ is almost of stable rank $1$. 
More generally, almost local-global rings in the sense of Couchot \cite{Cou07} are almost of stable rank $1$.

\begin{corollary} \label{CorCompleteInvariant1}
Let $R$ be a ring which is almost of stable rank $1$. 
Let $I$ be an ideal generated by two regular elements. 
Then the determinant map induces a bijection from $\V_2(I)/\SL_2(R)$ onto $\V_1(\det(I)) \simeq (R/\Fitt_1(I))^{\times}$.
\end{corollary}

\begin{proof}
It suffices to check that the two conditions of Proposition \ref{PropSL2CompleteInvariant} hold true.
By Lemma \ref{LemFitt1}, we have $\Fitt_1(I) = II^{-1}$. 
By assumption, the ring $R/(a: b)$ has stable rank $1$ and thus enjoys property $(*)$. As a result, the natural map 
$\left(R/(a: b)\right)^{\times} \rightarrow (R/\Fitt_1(I))^{\times}$ is surjective, which completes the proof.
\end{proof}

\begin{corollary} \label{CorFreeModule}
Let $R$ be an integral domain which is almost of stable rank $1$. 
Assume moreover that $R$ is a free module of rank $2$ over a K-Hermit subring $S$.
Let $I$ be a two-generated ideal of $R$. Then $\SL_n(R)$ acts transitively on $\V_n(I)$ for every $n > 2$.
\end{corollary}

\begin{proof}
We fix $(a, b)  \in \V_2(I)$ and consider an arbitrary generating vector $\mb \in \V_n(I)$ for $n > 2$.
As $R$ is a free $S$-module of rank $2$, Lemma \ref{FreeOverKHermite} applies. Hence we can find $\sigma \in \SL_n(R)$ such that  
$\mb \sigma = (\mb', 0, \dots, 0)$ for some $\mb' \in \V_2(I)$. By Lemma  \ref{LemSurjectivity} and its proof, we can find $\sigma' \in \SL_2(R)$ and $r \in R$ such that 
$\mb' \sigma' = (a, rb)$. Thus we have $(\mb', 0, \dots, 0) \sim_{\SL_n(R)} (a, rb, 0, \dots, 0) \sim_{\E_n(R)} (a, rb, b, 0, \dots, 0) \sim_{\E_n(R)} (a, b, 0, \dots, 0)$. 
Therefore $\mb \sim_{\SL_n(R)} (a, b, 0, \dots, 0)$.
\end{proof}

We denote by $\Min(R)$ the \emph{minimal prime spectrum} of $R$, that is, the set of minimal prime ideals of $R$. It is well-known that a 
Noetherian ring has a finite minimal prime spectrum \cite[Theorem 3.1]{Eis95} and this holds more generally for rings with the ascending condition on radical ideals. 
Let $R$ be a reduced ring with minimal prime spectrum. Let $I$ be a faithful ideal of $R$ generated by $\mu$ elements.
We shall establish that $I$ can be generated by $\mu$ regular elements. Combining this fact with Corollary \ref{CorCompleteInvariant1} above, we obtain:

\begin{corollary} \label{CorCompleteInvariant2}
Assume that $\Min(R)$ is finite and that $R$ is reduced and almost of stable rank $1$. Let $I$ be a two-generated faithful ideal of $R$. 
Then the determinant map induces a bijection from $\V_2(I)/\SL_2(R)$ onto $\V_1(\det(I)) \simeq (R/\Fitt_1(I))^{\times}$.
\end{corollary}

The following lemma resolves our debt with the proof of Corollary \ref{CorCompleteInvariant2}. 

\begin{lemma} \label{LemmaReduced}
Let $R$ be a reduced ring.
Then the following hold.
\begin{itemize}
\item[$(i)$] The set of zero divisors of $R$ is the union of the minimal prime ideals of $R$.
\item[$(ii)$] The annihilator of an ideal of $R$ is the intersection of the minimal prime ideals that do not contain it.
\item[$(iii)$] If $R$ is Noetherian then the set $\Ass(R)$ of associated primes of $R$ 
and the set $\Min(R)$ of minimal prime ideals of $R$ are finite and coincide.
\item[$(iv)$] If $\Min(R)$ is finite and $I$ is a faithful ideal generated by $\mb \in \V_n(I)$, then there is $\mb' \in \V_n(I)$, 
every component of which is regular, and such that $\mb \sim_{\E_n(R)} \mb'$.
\end{itemize}
\end{lemma}

\begin{proof}
$(i)$. See for instance \cite[Proposition 1.1]{Mat83}.
$(ii)$. This is routine.
$(iii)$. As the set of zero divisors of $R$ is the union of the associated primes of $R$ \cite[Theorem 3.1.b]{Eis95}, 
the result follows from $(i)$ together with prime avoidance \cite[Lemma 3.3]{Eis95}.
$(iv)$. Let $E$ be a subset of $\Min(R)$ such that 
$\bigcap_{\ip \in E} \ip = \{ 0 \}$ and $\bigcap_{\ip \in F} \ip \neq  0$ for every proper subset $F \subsetneq E$.
For each $\ip$ in $E$, we pick 
$r_{\ip} \in \left(\bigcap_{\iq \in E, \iq \neq \ip} \iq \right) \setminus \{ 0 \}$.
If $r$ is a regular element, then $r r_{\ip} \neq 0$ for every 
$\ip \in E$. Hence $r$ does not belong to any $\ip$ in $E$. 
Thus $\bigcup_{\ip \in E} \ip$ is the set of zero-divisors of $R$ and it follows from $(i)$ and prime avoidance \cite[Lemma 3.3]{Eis95} that $E = \Min(R)$.
Let $I$ be a faithful ideal of $R$ that is generated by $n$ elements, say $r_1, \dots, r_n$. Let $\ip \in \Min(R)$. 
Then at least one generator, say $r_1$, does not belong to $\ip$, since otherwise $r_{\ip}I = 0$ would hold with $r_{\ip}$ as above.
 Replacing each $r_i$ that lies in $\ip$ by $r_i + r_{\ip}r_1$ yields a new generating vector with no component in $\ip$. 
 By iterating this process over $\Min(R)$, we eventually obtain a vector of $n$ regular generators of $I$.
\end{proof}

We conclude this section with two examples of Bass rings for which 
the determinant map provides a particularly simple description of $\V_2(I)/\SL_2(R)$ for a faithful ideal $I$ of $R$. 

\begin{proposition} \label{PropFitt1Examples}
\begin{itemize}
\item[$(i)$] 
Let $R$ be the integral group ring of a cyclic group $\langle c \rangle$ with $p$ elements, $p > 0$ a prime number. 
Let $I$ be a faithful ideal of $R$ which is not invertible. Then $\Fitt_1(I) = R(c - 1) + Rp$ and $\V_2(I) / \SL_2(R) \simeq (\Z/p\Z)^{\times}$.
\item[$(ii)$]
Let $K$ be field of characteristic distinct from $2$. Let $$R = K[x, y] / K[x, y](y^2 - x^2(x + 1)),$$ i.e.,  
the coordinate ring of the nodal plane cubic curve $y^2 = x^2(x + 1)$ over $K$. 
Let $I$ be a non-zero ideal of $R$ which is not invertible. 
Then $\Fitt_1(I) = R\overline{x} + R\overline{y}$  and $\V_2(I)/\SL_2(R) \simeq K^{\times}$ 
where $\overline{x}$ and $\overline{y}$ are the images of $x$ and $y$ in $R$.
\end{itemize}
\end{proposition}

The proof of Proposition \ref{PropFitt1Examples} relies on the following two lemmas.

\begin{lemma} \label{LemProductOfDedekindRings}
Let $R$ be a reduced ring with finite minimal spectrum. 
Let $$S = \Pi_{\ip \in \Min(R)} R/\ip$$ and assume that each factor $R/\ip$ is a Dedekind ring. 
Let $I$ be a two-generated ideal of $R$.
Then the two following hold:
\begin{itemize}
\item[$(i)$] $\Fitt_0(S/R) \subseteq \Fitt_1(I)$. 
\item[$(ii)$] If $S/R$ is a cyclic $R$-module, then $\Fitt_0(S/R) = (R:S)$
\end{itemize}
\end{lemma}

\begin{proof}
$(i)$. For $\ip \in \Min(R)$, denote by $\ip'$ the product of the minimal primes of $R$ distinct from $\ip$. 
It is easy to see that $\Fitt_0(S/R) =  \sum_{\ip \in \Min(R)} \ip'$ by considering a presentation of the $R$-module $S/R$ 
whose generators are the obvious idempotents splitting $S$.
Since $\Fitt_i$ commutes with localization, it suffices to show that 
$\Fitt_0(S_{\iq}/R_{\iq}) \subseteq \Fitt_1(I_{\iq})$ for every maximal ideal $\iq$ of $R$. 
Let $\iq$ be such a maximal idea Then either $(R/\ip)_{\iq} = \{0\}$ or $\ip \subseteq \iq$ in which case $(R/\ip)_{\iq}$ is a local Dedekind ring, 
hence a principal ideal domain. Thus we can assume, without loss of generality 
that $S$ is a product of principal ideal domains.
Let $(a, b) \in \V_2(R)$ and $\ip \in \Min(R)$. 
As $R/\ip$ is a principal ideal domain, there is $E \in \E_2(R)$ such that $(a_{\ip}, b_{\ip})E = (c, 0)$ where 
$a_{\ip}$ and $b_{\ip}$ are the components of $a$ and $b$ in $R/\ip$ and $c \in R/\ip$.
Let $(a', b') = (a, b)E$. Then we have $\ip' \subseteq (a': b') \subseteq \Fitt_1(I)$, so that $\Fitt_0(S/R) \subseteq \Fitt_1(I)$.

$(ii)$. This follows from \cite[Proposition 20.7]{Eis95}. 
\end{proof}

The next lemma comes in handy for the ring of Proposition \ref{PropFitt1Examples}.$ii$ but also in Section \ref{SecCurves}.  

\begin{lemma} \label{LemPID}
Let $R, S$ be rings such that $R \subseteq S \subseteq K(R)$ 
and assume that $S$ is a principal ideal domain.
Let $I$ be a two-generated ideal of $R$.
Then $\Fitt_1(I)$ contains $(R: S)$.
\end{lemma}

\begin{proof}
We can assume that $I \neq \{0\}$ since the result is trivial otherwise. 
By hypothesis, there are $d, d' \in S$ such that  $(R:S) = Sd$ and $SI = Sd'$. 
Let $a$ and $b$ be two generators the fractional $R$-ideal $I/d'$. 
As $Sa + Sb = S$, it follows that $(a: b) + (b: a)$ 
is generated by the elements of the form $sa$ and $sb$ where $s \in S$ is such that $sa$ and $sb$ belong to $R$. 
Since $(R : S) = Sda + Sdb$, we infer that $(R: S) \subseteq \Fitt_1(I/d') = \Fitt_1(I)$.
\end{proof}

The \emph{normalization} $\tilde{R}$ or $R$ is the normal closure of $R$ in $K(R)$.

\begin{sproof}{Proof of Proposition \ref{PropFitt1Examples}}
$(i)$. 
The ring $R$ is a Dedekind-like ring \cite[Theorem 1.2]{Lev85}, hence a Bass ring (alternatively use the characterization \cite[Theorem 2.1.$iii$]{LW85}). Thus $I$ can be generated by two elements.
The normalization $\tilde{R}$ of $R$ identifies with $\Z \times \Z[e^{2  i \pi / p}]$ and $R$ embeds into $\tilde{R}$ through the map 
$r \mapsto (r + R(c - 1), r + R\Phi_p(c))$ where $\Phi_p(x) = 1 + x + \cdots + x^{p - 1}$ is the $p$th cyclotomic polynomial. From now on, we identifies $R$ with its image in $\tilde{R}$. 
By Lemma \ref{LemmaReduced}, the ideal $I$ can be generated by two regular elements. By Lemma \ref{LemFitt1}, we have $\Fitt_1(I) = II^{-1}$.
By Lemma \ref{LemProductOfDedekindRings}, the ideal  $\Fitt_1(I)$ contains $(R : \tilde{R})$, which is the maximal ideal $R(c - 1) + R\Phi_p(c) = R(c - 1) + Rp$. Since $I$ is not invertible, we have  
$\Fitt_1(I) = R(c - 1) + Rp$. Finally, the bijection $\V_2(I)/\SL_2(R) \simeq (\Z/p\Z)^{\times}$ results from Corollary \ref{CorCompleteInvariant2}.

$(ii)$. 
The ring $R$ is Bass ring \cite[Theorem 2.1]{Grei82}, hence $I$ can be generated by two elements.
It is easily checked that the map 
$x \mapsto z^2 - 1, y \mapsto z(z^2 - 1)$ 
induces a ring isomorphism from $R$ onto $K[z^2 -1, z(z^2 - 1)] \subsetneq K[z] \Doteq \tilde{R}$ and with this identification $\tilde{R}$ is the normalization of $R$. 
Since $\tilde{R}$ is a principal ideal ring and $I \neq \{0\}$, 
we have $(R: \tilde{R}) = K[z](z^2 - 1) \subseteq \Fitt_1(I)$ by Lemma \ref{LemPID}. 
As $K[z](z^2 - 1)$ is maximal in $R$ and $I$ is not invertible, we infer that 
$\Fitt_1(I) =  K[z](z^2 - 1) = R(z^2 -1) + Rz(z^2 - 1)$. 
Finally, the bijection $\V_2(I)/\SL_2(R) \simeq K^{\times}$ results from Corollary \ref{CorCompleteInvariant1}.
\end{sproof}

\begin{remark}
\begin{itemize}
\item
With the notation of Proposition \ref{PropFitt1Examples}.$i$, a non-faithful ideal $I$ of $R$ is annihilated by one of the two minimal prime ideals of $R$, say $\ip$. 
Hence such an ideal $I$ identifies with an ideal of $R/\ip \simeq \Z$ or $\Z[e^{2i\pi / p}]$. 
The natural map $\SL_2(R) \rightarrow \SL_2(R/\ip)$ is surjective by Proposition \ref{PropStableRank}.$ii$ and Remark \ref{RemStableRank2}.
It follows that $\SL_2(R)$ acts transitively on $\V_2(I)$. 
\item
For both assertions of Proposition \ref{PropFitt1Examples}, we have actually shown that $\Fitt_1(I) = (R: \tilde{R})$ provided that $I$ is faithful.
\end{itemize}
\end{remark}

\section{The group $\Aut_R(I)$ and the line $\PL(K(R))$} \label{SecP1}

This section is dedicated to the proof of Corollary \ref{CorP1}. 
We proceed by proving first a result valid in any ring, namely Proposition \ref{PropP1} below.
To do so, we consider the diagonal action of the $R$-automorphism group $\Aut_R(I)$ of $I$ from the left:
$$
\phi (a, b) \Doteq (\phi(a), \phi(b)), \, (a,b) \in \V_2(I).
$$
This action clearly commutes with the action of $\GL_2(R)$ from the right.
The multiplication by a unit $u \in \varrho(I)^{\times}$ gives rise to an $R$-automorphism of $I$
so that  $ \varrho(I)^{\times}$ naturally identifies with a subgroup of $\Aut_R(I)$.

We define an action of $\SL_2(R)$ on $\PL(K(R))$ from the right in the following way:
$$
[r:s] \sigma = [ar + cs: br + ds]
$$
with $\, [r:s] \in \PL(K(R))$ and 
$\sigma =  \begin{pmatrix} a & b \\ c & d \end{pmatrix} \in \SL_2(R)$. 

\begin{proposition} \label{PropP1}
Let $R$ be a ring. 
Then $\PL(K(R)) /\SL_2(R)$ is equipotent with the disjoint union
$$
\bigsqcup_{[I] \in \Cl(R), \, \mu(I) \le 2} \varrho(I)^{\times} \backslash \V_2(I) /\SL_2(R) 
$$
\end{proposition}

\begin{proof}
Let $\psi: \PL(K(R)) /\SL_2(R) \rightarrow \{ [I] \in \Cl(R) \, \vert \, \mu(I) \le 2 \}$ be defined by 
$[a: b]\SL_2(R) \mapsto [Ra + Rb]$. The map $\psi$ is obviously surjective. Hence it only remains to show that 
$\psi^{-1}(\{[I]\})$ is equipotent with $\varrho(I)^{\times} \backslash \V_2(I) /\SL_2(R)$ for every ideal class $[I]$ such that $\mu(I) \le 2$. 
We can assume that $I \subseteq R$. Moreover, every element in $\psi^{-1}(\{[I]\})$, can be represented by a coset of the form $[a:b]\SL_2(R)$ with $(a, b) \in \V_2(I)$. 
We claim that the map $\theta: [a:b]\SL_2(R) \rightarrow \varrho(I)^{\times} (a, b) \SL_2(R)$ is a well-defined bijection from $\psi^{-1}(\{[I]\})$ 
onto $\varrho(I)^{\times} \backslash \V_2(I) /\SL_2(R)$. 
Indeed, the identity $[a:b]\SL_2(R) = [a':b']\SL_2(R)$ holds for $(a, b), (a', b') \in \V_2(I)$ if and only if there is $u \in K(R)$ and $\sigma \in \SL_2(R)$ such that 
$(a', b') = (ua, ub)\sigma$. Under this assumption, we clearly have $u \in \varrho(I)^{\times}$, thus $\theta$ is well-defined. Proving that $\theta$ is bijective is routine.
\end{proof}

\begin{lemma} \label{LemDetU}
Let $I$ be an ideal of $R$ and let $u \in \varrho(I)^{\times}$. Then 
$\det_{\mb}(u \mb) = \det_{\mb'}(u \mb')$ for every $\mb, \mb' \in \V_2(I)$.
\end{lemma}

\begin{proof}
Let $\mb = (a, b), \mb' = (a', b') \in \V_2(I)$.
By Lemma \ref{LemDet}.$ii$, we have 
$$\Det_{(a', b')}(ua', ub') = \Det_{(a', b')}(a, b)\Det_{(a, b)}(ua, ub) \Det_{(ua, ub)}(ua', ub').$$
It follows straightforwardly from Lemma \ref{LemDetPi} 
that $$\Det_{(ua, ub)}(ua', ub') = \Det_{(a, b)}(a', b') =  \Det_{(a', b')}(a, b)^{-1},$$ hence the result. 
\end{proof}

\begin{definition} \label{DefEpsilonI}
For $u \in \varrho(I)^{\times}$, we denote by $\det(u)$ the value of $\det_{\mb}(u \mb)$ for any, and hence all generating pairs $\mb$ of $I$.
We denote by $\epsilon(I)$ the image of $\varrho(I)^{\times}$ by $\phi_{\mb}$ (see Section \ref{SubSecFittAndDet}) for any, and hence every $\mb \in \V_2(I)$, that is 
$$\epsilon(I) \Doteq \{ \det(u) \, \vert \, u \in \varrho(I)^{\times} \}.$$
\end{definition}

By Lemma  \ref{LemDetPi} , we have $\det(1) = 1 + \Fitt_1(I)$, $\det(uu') = \det(u) \det(u')$ for every $u, u' \in \varrho(I)^{\times}$
and $\det(u) = u^2 + \Fitt_1(I)$ for every $u \in R^{\times}$. 
Therefore $\epsilon(I)$ is a subgroup of $(R/\Fitt_1(I))^{\times}$ containing the image of $(R^{\times})^2$.

\begin{proposition} \label{PropEpsilonI}
Let $I$ be a two-generated ideal of $R$ such that $\Fitt_1(R) = II^{-1}$. Then the determinant map induces an injective map
$$\varrho(I)^{\times} \backslash \V_2(R) / \SL_2(R) \rightarrow \varrho(I)^{\times} \backslash \V_1(\det(I)) \simeq (R/\Fitt_1(I))^{\times} / \epsilon(I).$$
If $I$ satisfies the hypotheses of Lemma \ref{LemSurjectivity}, then the induced map is a bijection.
\end{proposition}

\begin{proof}
Apply Theorem \ref{ThSL2AndFitt1I} and Lemma \ref{LemDetU}.
\end{proof}

\begin{sproof}{Proof of Corollary \ref{CorP1}}
Combine Propositions \ref{PropP1} and \ref{PropEpsilonI}, observing that the hypotheses of 
Lemma \ref{LemSurjectivity} are satisfied. Indeed, an order in a number field is a one-dimensional domain, hence almost of stable range $1$.
\end{sproof}

\section{Equivalence of generating vectors in Bass rings} \label{SecBassRings}
This section illustrates Theorem \ref{ThSL2} with two examples of Bass rings. 
Subsection \ref{SecOrders} deals with arbitrary orders in quadratic number fields. 
Subsection \ref{SecCurves} deals with the coordinate ring of the plane curve $y^2 = x^n$ where $n \ge 3$ is an odd rational integer.

Recall that the normalization $\tilde{R}$ of $R$ is the integral closure of $R$ in its total ring of quotients $K(R)$.
A \emph{Bass ring} is a Noetherian reduced ring whose normalization $\tilde{R}$ 
is a finitely generated $R$-module and whose ideals can be generated by two elements. 
Bass domains were originally studied by Bass who proved that their finitely generated torsion-free modules split into a direct product of ideals and that the converse 
holds for a Noetherian domain $R$ if $\tilde{R}$ is a finitely generated $R$-module \cite[Theorem 1.7]{Bass62}, see also \cite[Section 4]{Lam00} for historical remarks.

For every example of Bass ring in this article, we observe that $\Fitt_1(I)$ contains $(\tilde{R} : R)$ if $I$ is a faithful ideal of $R$. 
For each example, this implies that $\Fitt_1(I)$ takes only finitely many values because of

\begin{proposition} \cite[Proposition 2.2]{LW85}
Let $R$ be a Bass ring. Then $R/(R: \tilde{R})$ is an Artinian principal ideal ring.
\end{proposition}

\subsection{Ideals in orders of quadratic fields} \label{SecOrders}
In this section, we prove Theorem \ref{ThOrder}. On our way, we compute $\Fitt_1(I)$ and 
describe $\V_2(I)/\SL_2(I) \simeq (R/\Fitt_1(I))^{\times}$ explicitly for an arbitrary ideal $I$ of an order in a quadratic number field. 
 
Let $m$ be a square-free rational integer. We set $K \Doteq \Q(\sqrt{m})$ and denote by $\iO$ the ring of integers of the quadratic field $K$. 
An \emph{order} of $K$ is a subring of finite index in $\iO$. Setting $\omega = \left\{
\begin{array}{cc}
\sqrt{m} & \text{ if } m \not\equiv 1 \mod 4, \\
\frac{1 + \sqrt{m}}{2} & \text{ if }  m \equiv 1 \mod 4. \\
\end{array}\right.
$, then we have $\iO = \Z + \Z \omega$ and any order of $K$ is of the form $\iO_f \Doteq \Z + \Z f \omega$ for some rational integer $f > 0$ \cite[Lemma 6.1]{IK14}. 
Moreover, the inclusion $\iO_f \subseteq \iO_{f'}$ holds true if and only if $f'$ divides $f$.
If $I$ is an ideal of $\iO_f$, then its ring of multipliers $\varrho(I) \Doteq \{ r \in \iO \, \vert \, rI \subseteq I\}$ is 
the smallest order $\iO$ of $K$ such that $I$ is projective, equivalently invertible, as an ideal of $\iO$ \cite[Proposition 5.8]{LW85}.
For the remainder of this section, we fix $f > 0$ and set $R \Doteq  \iO_f$ and $\tilde{R} \Doteq \iO$.

\begin{proposition} \label{PropFitt1OfOrderIdeal}
Let $I$ be an ideal of $R$ and let $f' > 0$ be such that $\varrho(I) = \iO_{f'}$. 
Then we have
\begin{itemize}
\item[$(i)$] $\Fitt_1(I) = (R : \varrho(I))$,
\item[$(ii)$] $R / \Fitt_1(I) \simeq \varrho(I) / R$ as $R$-modules and $R / \Fitt_1(I)  \simeq \Z / (f/f') \Z$ as rings.
\end{itemize}
\end{proposition}

\begin{sproof}{Proof of Proposition \ref{PropFitt1OfOrderIdeal}.$ii$}
By assertion $(i)$,  we have $R / \Fitt_1(I) = R / (R : \varrho(I))$. 
As $\Rt / R \simeq \Z /f \Z$, the $R$-submodule $\varrho(I) / R$ of $\Rt / R$ is cyclic so that $R / (R : \varrho(I)) \simeq \varrho(I) / R$. 
Since the image of $\varrho(I)/R$ in $\Z /f \Z \simeq \Rt / R$ is $f'(\Z/f\Z)$,  assertion $(ii)$ is established.
\end{sproof}

Our proof of Proposition \ref{PropFitt1OfOrderIdeal}.$i$ relies on 
three lemmas revolving around the \emph{standard basis} of an ideal, a definition that we introduce after the next lemma.
We define the \emph{norm $\Norm(I)$ of an ideal $I$ of $R$} as the index $[R : I] = \vert R / I \vert$ of $I$ in $R$. 
An ideal $I$ of $R$ is said to be \emph{primitive} if it cannot be written as $I = eJ$ some rational integer $e$ and some ideal $J$ of $R$. 

\begin{lemma}\cite[Lemma 6.2 and its proof]{IK14} \label{LemStandardBasis}
Let $I$ be a non-zero ideal of $R = \iO_f$. Then there exist rational integers $a, e > 0$ and $d \ge 0$ such that $-a/2 \le d < a/2$, 
$e$ divides both $a$ and $d$ and we have
$$
I = \Z a + \Z(d + e f \omega).
$$
The integers $a, d$ and $e$ are uniquely determined by $I$. The integer $ae$ is the norm of $I$ and we have $\Z \cap I = \Z a$. The ideal $I$ is primitive if and only if $e = 1$
\end{lemma}
Note that, since  $\Z \cap I = \Z a$, the rational integer $a$ divides $\Norm(d + e f \omega)$. 
We call the generating pairs $(a, d + ef \omega)$ the \emph{standard basis of $I$}. 
Let us associate to $I$ the binary quadratic form $q_I$ defined by $$q_I(x, y) = \frac{\Norm(xa + y(d + ef\omega))}{\Norm(I)}.$$

Then we have 
$$e q_I(x, y) = ax^2 + bxy + cy^2$$ 
with $$b = \Tr(d + ef \omega) \text { and } c = \frac{\Norm(d + ef \omega)}{a}.$$
We define the \emph{content of $q_I$} as the greatest common divisors of its coefficients, and we write $$\content(q_I) \Doteq \frac{\gcd(a, b, c)}{e}.$$

\begin{lemma} \label{LemContent}
Let $I$ be a non-zero ideal of $R$ with standard basis $(a, d + ef \omega)$. Let $f' > 0$ be such that $\varrho(I) = \iO_{f'}$. 
Then we have $$\gcd(a, d, ef) = e \content(q_I) = ef/f'.$$
\end{lemma}

\begin{sproof}{Proof of Lemma \ref{LemContent}}
By Lemma \ref{LemStandardBasis}, the fractional ideal $J \Doteq \frac{1}{e}I$ is an integral ideal of $R$ with standard basis $(a/e, d/e + f\omega)$. 
It is immediate to see that the result holds for $I$ if and only if it holds for $J$. Therefore, we can assume, without loss of generality, that $e = 1$. 
Now the result can be inferred straightforwardly from \cite[Proof of Lemma 6.5]{IK14}.
\end{sproof}

We denote by $\sigma$ the automorphism of $K$ which maps $\sqrt{m}$ to $-\sqrt{m}$.

\begin{lemma} \label{LemmaIMinus}
Let $I$ be a non-zero ideal of $R$ and let $f' > 0$ be such that $\varrho(I) = \iO_{f'}$. 
Then the two following hold.
\begin{itemize}
\item[$(i)$] $I^{-1} = \frac{1}{\Norm(I)}\sigma(I)$,
\item[$(ii)$] $\Fitt_1(I) = II^{-1} = \Z\frac{f}{f'} + \Z f \omega$.
\end{itemize} 
\end{lemma}

Assertion $(i)$ of Lemma \ref{LemmaIMinus} is usually stated and proven in the case of an invertible ideal $I$ of $R$. 
We provide a proof of the general case for the convenience of the reader.

\begin{sproof}{Proof of Lemma \ref{LemmaIMinus}}
We consider the standard basis $(a, d + ef \omega)$ of $I$.
By Lemma \ref{LemStandardBasis}, we can write $I = eJ$ for some primitive ideal $J$ of $R$. We claim that if assertions $(i)$ and $(ii)$ are satisfied by $J$, 
then they are satisfied by $I$. Indeed, it is trivial to check that $(eJ)^{-1} = \frac{1}{e}J^{-1}$, $\Norm(eJ) = e^2 \Norm(J)$ and $\sigma(eJ) = e \sigma(J)$.
These three facts settle the claim. Hence we can assume, without loss of generality, that $e = 1$. 

$(i)$. By definition, we have $\frac{1}{\Norm(I)}\sigma(I) = \Z + \Z \frac{d + f\sigma(\omega)}{a}$ and the inclusion $\frac{1}{\Norm(I)}\sigma(I) \subseteq I^{-1}$ is evident. 
In order to prove the reverse inclusion, let us take $x \in I^{-1}$. Since $xa \in R$, we can write $x = \frac{u + v f \omega}{a}$ for some $u,v \in \Z$. 
Let us conveniently rewrite $x$ as $x = z - v \frac{d + f\sigma(\omega)}{a}$ with $z \Doteq \frac{u + vf \Tr(\omega) + vd}{a}$. Since $x (d + f \omega) \in R$, 
we deduce that $z(d + f \omega) \in R$. The elements $1$ and $f \omega$ are linearily independent over $\Q$ and freely generate the ring $R$ as a $\Z$-module. 
Therefore  $z$ lies into $\Z$, which entails $x \in \Z + \Z\frac{d + f\sigma(\omega)}{a}$, as desired.

$(ii)$. Using the formula of assertion $(i)$, we obtain $II^{-1} = \Z a + \Z b + \Z c + \Z(d + f\omega)$.
Applying Lemma \ref{LemContent}, we get $II^{-1} = \Z\frac{f}{f'}  + \Z(d + f\omega) = \Z\frac{f}{f'}  + \Z f\omega$. 
\end{sproof}

We are now in position to complete the proof of Proposition \ref{PropFitt1OfOrderIdeal}.

\begin{sproof}{Proof of Proposition \ref{PropFitt1OfOrderIdeal}.$i$}
As $R$ is a Bass domain, it follows from Lemma \ref{LemFitt1} that $\Fitt_1(I) = II^{-1}$. 
From the inclusion $\varrho(I)II^{-1} \subseteq R$, we infer that $\Fitt_1(I) \subseteq (R : \varrho(I))$. 
We have shown in the first part of the proof of \ref{PropFitt1OfOrderIdeal}, that the index of $(R : \varrho(I))$ in $R$ is $f/f'$. 
The index of $\Fitt_1(I)$ is also $f/f'$ by Lemma \ref{LemContent}. Therefore $\Fitt_1(I) = (R : \varrho(I))$.
\end{sproof}

\begin{proposition} \label{PropENWithNGe3}
Let $I$ be an ideal of $R$. Then $\glr(I) = 2$.
\end{proposition}

\begin{proof}
As $R$ is a Bass ring, we have $\mu(I) = 2$. 
Moreover $R$ is a one-dimensional domain, hence it is almost of stable rank $1$. 
Observing eventually that $R$ is a free $\Z$-module of rank $2$, we can apply Corollary \ref{CorFreeModule}.
\end{proof}

We conclude this section with the proof of Theorem \ref{ThOrder} and a closely related result 
on the quotients $\varrho(I)^{\times} \backslash \V_2(I) /\SL_2(R)$ and $\V_2(I) /\GL_2(R)$, namely Proposition \ref{PropGL2} below.
Both results use

\begin{lemma} \label{LemNormU}
Let $I$ be an ideal of $R$. 
Then $\epsilon(I) = \{ \nK(u) + \Fitt_1(I) \, \vert \, u \in \varrho(I)^{\times}\}$ where $\nK(u)$ denotes the norm of $u$ over $\Q$. 
In particular, $\epsilon(I)$ has at most two elements.
\end{lemma}

\begin{proof}
Let $\mb \in \V_2(I)$ and let $u \in \varrho(I)^{\times}$. Let $A \in \GL_2(\Z)$ be such that $u \mb = \mb A$. 
By the definitions of $\nK(u)$ and of $\det(u)$ we have $\nK(u) = \det(A)$ and $\det(u) = \det(A) + \Fitt_1(I)$, hence the result.
\end{proof}

\begin{sproof}{Proof of Theorem \ref{ThOrder}}
Let $\kappa(R)$ be the number of cusps of $R$.
By Propositions \ref{PropP1} and \ref{PropEpsilonI}, we have $\kappa(R) = \sum_{[I] \in \Cl(R)} [(R/\Fitt_1(I))^{\times} : \epsilon(I)]$. 
We also have $\Fitt_1(R) = (R: \varrho(I))$ by Proposition \ref{PropFitt1OfOrderIdeal}.
This identity together with Definition \ref{DefEpsilonI} imply that 
$[(R/\Fitt_1(R))^{\times} : \epsilon(I)]$ depends only on $\varrho(I)$. 
Since $\varrho(I)$ is the smallest overring $\iO_{f'}$ of $R$ such that $I$ is an invertible ideal of $\iO_{f'}$ \cite[Proposition 5.8]{LW85}, we have 
$\kappa(R) = \sum_{f' \vert f}  \vert \Pic(\iO_{f'}) \vert [(R/(R: \iO_{f'}))^{\times} : U_{f'}]$ where $U_{f'}$ is described by Lemma \ref{LemNormU}. 
By Lemma \ref{LemmaIMinus}, we have $R/(R: \iO_{f'}) \simeq \Z /(f/f') \Z$. Since $U_{f'}$ has at most two elements and exactly two if $\iO_{f'}$ 
has a unit of norm $-1$ and $f/f' > 2$, we infer that 
$[(R/(R: \iO_{f'}))^{\times} : U_{f'}] = \frac{\varphi(f/f')}{2^{\epsilon(f, f')}}$, which completes the proof.
\end{sproof}

\begin{proposition} \label{PropGL2}
Let $I$ be a non-zero ideal of $R$ and let $f'$ be the positive rational integer such that $\varrho(I) = \iO_{f'}$. 
Then the following hold.
\begin{itemize}
\item[$(i)$]
$\vert \varrho(I)^{\times} \backslash \V_2(I) /\SL_2(R) \vert = [(R/\Fitt_1(I))^{\times} : \epsilon(I)] =  \frac{\varphi(f/f')}{2^{\epsilon(f, f')}}$ with $\epsilon(f, f')$ as in Theorem \ref{ThOrder}.
\item[$(ii)$]
We have $\Aut_R(I) \backslash \V_2(I) /\GL_2(R) = \V_2(I) /\GL_2(R)$. The determinant map induces a bijection of $\V_2(I) /\GL_2(R)$ 
onto the quotient of $(R/\Fitt_1(I))^{\times}$ by the image of $R^{\times}$
 through the natural map $R \rightarrow R/\Fitt_1(I)$.
\item[$(iii)$]
If $f/f'$ has at least three distinct odd rational prime divisors, then $\GL_2(R)$ doesn't act transitively on $\V_2(I)$.
\end{itemize}
\end{proposition}

\begin{proof}
Assertion $(i)$ has been already established in the course of the proof of Theorem \ref{ThOrder}. 

$(ii)$. 
Since $I$ is an invertible ideal of $\varrho(I)$, every $R$-automorphism of $I$ is induced by the multiplication by a unit of $\varrho(I)$. 
Thus we have $$\Aut_R(I) \backslash V_2(I) /\GL_2(R) = \varrho(I)^{\times} \backslash \V_2(I) /\GL_2(R)$$ 
and the identity $\varrho(I)^{\times} \backslash \V_2(I) /\GL_2(R) = \V_2(I) /\GL_2(R)$ follows from Lemma \ref{LemNormU}.
Consider now $\mb \in \V_2(I)$. Then for every $u \in R^{\times}$ we have 
$$\Det_{\mb}(\mb \begin{pmatrix} u & 0 \\ 0 & 1 \end{pmatrix}) = u + \Fitt_1(I).$$ 
Since the determinant map induces a bijection from 
$\V_2(R)/\SL_2(R)$ onto the group $(R/\Fitt_1(I))^{\times}$ by Corollary \ref{CorCompleteInvariant1}, it induces a bijection from 
$\V_2(I)/\GL_2(R)$ onto the quotient group $Q \Doteq (R/\Fitt_1(I))^{\times}/(R^{\times} + \Fitt_1(I))$.

$(iii)$. 
Since $(R/\Fitt_1(I))^{\times} \simeq (\Z/(f/f')\Z)^{\times}$, it follows easily from the Chinese Remainder Theorem that $(R/\Fitt_1(I))^{\times}$ surjects onto $(\Z/2\Z)^3$. 
Hence $(R/\Fitt_1(I))^{\times}$ cannot be generated by less than three elements. 
As $R^{\times}$ is generated by $-1$ and a power of the fundamental unit of $K$, the group $Q$ defined in the proof of $(ii)$ is not trivial.
\end{proof}

\subsection{Ideals in the coordinate ring of the curve $x^2 = y^n$} \label{SecCurves}
This section is dedicated to the proof of Proposition \ref{PropCurves} which illustrates  
Theorem \ref{ThSL2} in the coordinate ring $R$ of the curve $x^2 = y^n$ over $K$ with $n \ge 3$ an odd integer and $K$ an arbitrary field.
The coordinate ring of an algebraic curve is a Bass domain if and only if the only singularities of the curve are double points \cite[Theorem 2.1]{Grei82}.
This is well-known to hold for $R$. 
We identify $R$ with $K[x^2, x^n]$, a notation which stays in effect throughout this section. 
The integral closure of $R$ in its field of fractions is $K[x]$ and we have 
$R = K[x^2] + K[x] x^{n - 1} = K[x^2] + K[x^2]x^n$ and $(R:K[x]) = K[x]x^{n - 1}$, i.e., the conductor of $R$ in $K[x]$ is $K[x]x^{n - 1}$.

We compute first $\Fitt_1(I)$ for $I$ a non-zero ideal of $R$. This is achieved by means of the next two lemmas.
\begin{lemma} \label{LemStandardBasisCurves}
Let $I$ be a non-zero ideal of $R$ and let $d(x) \in K[x]$ be a generator of $K[x]I$. Then there is an even integer $\nu$ satisfying $0 \le \nu  \le n - 1$ 
and two coprime polynomials $p(x) = \sum_i p_i x^i, \,q(x) = \sum_i q_i x^i$ of $K[x]$ such that the following hold:
\begin{itemize}
\item[$(i)$] $R p(x) + R q(x) = \frac{1}{d(x)}I$,
\item[$(ii)$] $p_0 q_{\nu + 1} - p_1 q_{\nu} \neq 0$ if $\nu < n - 1$,
\item[$(iii)$] $p_0 \neq 0$ and $q_i = 0$ for every $i < \nu$. 
\end{itemize}
\end{lemma}

The \emph{$x$-valuation} of a polynomial $p(x) \in K[x]$ is the greatest integer $v \ge 0$ such that $x^{v}$ divides $p(x)$.

\begin{proof}
Let $p(x), p'(x)$ be the quotients of two generators of $I$ by $d(x)$. Then $p(x)$ and $p'(x)$ are coprime in $K[x]$. Swapping them if need be, we can suppose that $p(0) \neq 0$. 
Let $q(x) \in p(x) + Rp'(x)$ be a polynomial with maximal $x$-valuation. If $q(x) = 0$, the result is obvious. Thus we can assume that $q(x) \neq 0$.
Let $m$ be the greatest even integer such that $q_m$ or $q_{m + 1}$ is non-zero. Then we have $p_0 q_{m + 1} - p_1 q_m \neq 0$ since otherwise 
$q(x) - \lambda x^m p(x)$ would have an $x$-valuation greater than the $x$-valuation of $q(x)$ for some $\lambda \in K$. 
Setting $\nu = \min(m, n - 3)$, we easily check that $\nu, p(x)$ and $q(x)$ satisfy 
the conditions $(i)$, $(ii)$ and $(iii)$.
\end{proof}

\begin{lemma} \label{LemFittNu}
Let $I$ be an ideal of $R$ and let $p(x), q(x)$ and $\nu$ be as in Lemma \ref{LemStandardBasisCurves}. 
Then $$\Fitt_1(I) = K[x^2]x^{n - \nu - 1} + K[x] x^{n - 1}.$$
In particular $\nu$ is uniquely determined by $I$.
\end{lemma}

\begin{proof}
As $p(x)$ and $q(x)$ are coprime elements of $K[x]$, it easily follows from Lemma \ref{LemFittMuMinusOne} that $\Fitt_1(I)$ is 
generated by the elements of the form $h(x)p(x)$ and $h(x)q(x)$ with $h(x) \in K[x]$ such that
\begin{equation} \label{EqH}
h(x)p(x)\in R, h(x)q(x) \in R.
\end{equation}
Hence $\Fitt_1(I)$ contains the conductor $K[x]x^{n - 1}$ of $R$ in $K[x]$. Since 
$$R/K[x]x^{n - 1} \simeq K[x^2] / K[x^2]x^{n - 1}$$ is a local principal ideal ring 
whose maximal ideal is generated by the image of $x^2$, we infer that $\Fitt_1(I) = K[x^2]x^{\nu'} + K[x] x^{n - 1}$ for some even integer $\nu' \ge 0$. 
We claim that every polynomial $h(x)$ satisfying $(\ref{EqH})$ has an $x$-valuation at least $n - \nu - 1$ and  there is such an $h(x)$ with an $x$-valuation equal to $n - \nu - 1$. 
It immediately follows from the claim that $\nu' = n- \nu - 1$, hence the result.

In order to prove the claim, we consider a polynomial $h(x) = \sum_i h_i x^i \in K[x]$.
and let $\hb = (h_0, h_1, \dots, h_{n - 2})$. Then $h(x)$ satisfies $(\ref{EqH})$ if and only if $A\hb = 0$ where $A$ is the $(n - 1) \times (n - 1)$ matrix obtained by appending 
the last $\frac{\nu}{2}$ rows of the matrix
$
P = 
\begin{pmatrix}
p_1 & p_0 && &&&\\
p_3 & p_2  & p_1 & p_0 &&&\\
\vdots & \vdots & \vdots & \vdots &&&\\
p_{n - 2} & p_{n  - 3}  & p_{n - 4} & p_{n - 5} & \cdots & p_1 & p_0\\
\end{pmatrix}
$
to 
$$
Q =
\begin{pmatrix}
q_{\nu + 1} & q_{\nu} &&&&&&&\\
p_1 & p_0  && &&&&&\\
q_{\nu + 3} & q_{\nu + 2} & q_{\nu + 1}  & q_{\nu} &&&&&\\
p_3 & p_2  & p_1 &  p_0 &&&&&\\
\vdots & \vdots & \vdots & \vdots &&&&&\\
q_{n - 2} & q_{n - 3} & q_{n - 4} & q_{n - 5} & \cdots & q_{\nu + 1} & q_{\nu} & \\
p_{n - \nu -2} & p_{n - \nu - 3}  & p_{n - \nu - 4} & p_{n - \nu - 5} & \cdots & p_0  & p_1\\
\end{pmatrix}
$$
where every unspecified coefficient is zero.

Let $h(x)$ be such that $A \hb = 0$.
If $\nu < n - 1$, then $p_0 q_{\nu + 1} - p_1 q_{\nu} \neq 0$, so that $Q$ is invertible. Otherwise $Q$ is the $0 \times 0$ 
matrix and hence is invertible as well.  
As a result, we have $(h_0, \dots, h_{n - \nu - 2}) = 0$, i.e., the $x$-valuation of $h(x)$ is at least $n - \nu - 1$. 
The equation $A \hb = 0$ is equivalent to $(h_0, \dots, h_{n - \nu - 2}) = 0$ and
$A' \hb' = 0$ where $\hb' = (h_{n - \nu  - 1}, \dots, h_{n - 2})$ and $A'$ consists in the first $\frac{\nu}{2}$ rows of $P$. 
Since the latter equation clearly admits a solution $\hb'$ with $h_{n - \nu - 1} \neq 0$, 
we can find a polynomial $h(x)$ which satisfies $(\ref{EqH})$ and whose $x$-valuation is exactly $n - \nu - 1$.
\end{proof}

\begin{lemma} \label{LemRhoI}
Let $I$ be an ideal of $R$ and let $\nu$ be as in Lemma \ref{LemFittNu}. 
Then we have $\varrho(I) = K[x^2] + K[x]x^{\nu}$ and $\Fitt_1(I) = (\varrho(I) : R)$.
\end{lemma}

\begin{proof}
Let us prove first that every proper $R$-submodule of $K[x]$ containing $R$ is of the form $K[x^2] + K[x]x^{\nu'}$ for some even integer $\nu' > 0$. 
Clearly, there are exactly $\frac{n - 1}{2}$ modules of the latter form. Hence it suffices to show that $K[x] / R$ has length $\frac{n - 1}{2}$. 
This follows from the isomorphisms $K[x] /R \simeq R / (R: K[x]) \simeq K[x^2] / K[x^2] x^{n - 1}$. The claim is therefore established and we have
$\varrho(I) =  K[x^2] + K[x]x^{\nu'}$ for some even $\nu' \ge 0$. Clearly $\nu' + 1$ is the smallest odd integer $i > 0$ such that $x^i I \subset I$. 
To complete the proof, we show that $\nu + 1$ enjoys the same property. 
Let $\mathcal{B}$ be the $K$-basis of $K[x] / K[x]x^{n - 1}$ consisting of the images of $1, x, \dots, x^{n - 2}$.
We shall consider an $(n - 1) \times (n - 1)$ matrix $A$ over $K$ such that 
the two following are equivalent for every odd integer $i > 0$:
\begin{itemize}
\item $x^i p(x) \in I/d(x)$,
\item the linear system $A\hb = \bb_i$ has a solution $\hb \in K^{n - 1}$ 
where $\bb_i$ is the component vector of the image of $x^i p(x)$ in $K[x] / K[x]x^{n - 1}$ with respect to $\mathcal{B}$.
\end{itemize}

Let $A$ be the matrix whose columns are the component vectors of the images in $K[x] / K[x]x^{n - 1}$ of the polynomials
$$
\begin{array}{cc}
p(x), x^2p(x), \dots, x^{\nu - 2}p(x), \nonumber \\
x^{\nu}p(x), q(x), x^{\nu + 2}p(x), x^2q(x), \dots, x^{n - 3}p(x), x^{n - 3 - \nu}q(x) \nonumber
\end{array}
$$
with respect to $\mathcal{B}$.  
Since $I/d(x)$ contains $K[x]x^{n - 1}$, the matrix $A$ has the desired property. Moreover $A$ is of the form 
$
\begin{pmatrix}
A' & 0 \\
\ast & Q
\end{pmatrix}
$
with 
$$A' = 
\begin{pmatrix}
p_0 &&& &&\\
p_1 &&&&&\\
p_2 & p_0 &&&&\\
p_3 & p_1 &&&&\\
\vdots & \vdots & \ddots & \ddots &&\\
p_{\nu - 2} & p_{\nu  - 4}  & p_{\nu - 6} & \cdots &  \cdots  & p_0\\
p_{\nu - 1} & p_{\nu  - 3}  & p_{\nu - 5} & \cdots &  \cdots  & p_1\\
\end{pmatrix}
$$ 
and where $Q$ is a lower block-trigonal matrix with diagonal blocks all equal to the invertible matrix
$
\begin{pmatrix}
p_0 & q_{\nu} \\
p_1 & q_{\nu + 1}
\end{pmatrix}
$.
As $Q$ is invertible, the system $A \hb = \bb$ with $\hb = (h_0,\dots, h_{n - 2}), \bb = (b_0, \dots, b_{n - 2})$ has a solution if and only if $A' \hb' = \bb'$ has a solution where 
$\hb' = (h_0, \dots, h_{\nu - 1})$ and $\bb' = (b_0, \dots, b_{\nu - 1})$. 
Let $i > 0$ be an odd integer.
It is easy to check that the system $A'\hb' = \bb_i'$ has a solution if and only if $i \ge \nu + 1$. 
Let $\cb_{\nu + 1}$ be the component vector of the image of $x^{\nu + 1}q(x)$. It is also easy to check that $A\hb = \cb_{\nu + 1}$ has a solution, which means that 
$x^{\nu + 1}q(x) \in I/d(x)$. Therefore $x^i I \subset I$ holds true if and only if $i \ge \nu + 1$.
\end{proof}

\begin{proposition} \label{PropSL}
Let $I$ be an ideal of $R$. Then $\glr(I) = 2$.
\end{proposition}

\begin{proof}
Since $R$ is a Bass ring, we have $\mu(I) \le 2$.
As $R$ is a one-dimensional, it is almost of stable rank $1$. 
Since $R$ is moreover an integral domain which is a free $K[x^2]$-module of rank $2$, Corollary \ref{CorFreeModule} applies.
\end{proof}

\bibliographystyle{alpha}
\bibliography{Biblio}

\end{document}